\documentclass[11pt]{amsart}
\usepackage{CJK} % For Chinese characters
\usepackage{manfnt} % For cone (mancone) symbol
\usepackage[normalem]{ulem} % For crossing text out
\usepackage[makeroom]{cancel} % For crossing math out
\usepackage{hyperref}
\usepackage{amsfonts}
\usepackage{amsmath}
\usepackage{amssymb}
\usepackage{amscd}
\usepackage{graphicx}
\usepackage{enumitem}
\usepackage{latexsym}
\usepackage{color}
\usepackage{comment}
\usepackage{epsfig}
\usepackage{amsopn}
\usepackage{tikz-cd}
\usepackage{xypic}
\usepackage{mathrsfs}
\usepackage{array}
\usepackage[usenames,dvipsnames]{pstricks}
\usepackage{epsfig}
\usepackage{pst-grad} % For gradients
\usepackage{pst-plot} % For axes
\usepackage[space]{grffile} % For spaces in paths
\usepackage{etoolbox} % For spaces in paths
\makeatletter % For spaces in paths
\patchcmd\Gread@eps{\@inputcheck#1 }{\@inputcheck"#1"\relax}{}{}
\makeatother

\hypersetup{pdfpagemode=UseNone}
\usepackage{booktabs}
\usepackage{longtable}
\theoremstyle{plain}
\newtheorem{lemma}{Lemma}[section]
\newtheorem*{theorem*}{Theorem}
\newtheorem*{lemma*}{Lemma}
\newtheorem*{proposition*}{Proposition}
\newtheorem*{conjecture*}{Conjecture}
\newtheorem*{corollary*}{Corollary}
\newtheorem*{problem*}{Problem}
\newtheorem{theorem}[lemma]{Theorem}
\newtheorem{conjecture}[lemma]{Conjecture}
\newtheorem{corollary}[lemma]{Corollary}
\newtheorem{proposition}[lemma]{Proposition}

\theoremstyle{definition}
\newtheorem{definition}[lemma]{Definition}
\newtheorem{example}[lemma]{Example}

\newcommand{\CC}{\mathbb{C}}
\newcommand{\PP}{\mathbb{P}}
\newcommand{\OO}{\mathcal{O}}
\newcommand{\cO}{\mathcal{O}}

\newcommand{\mcH}{\mathcal{H}}
\newcommand{\cI}{\mathcal{I}}
\newcommand{\LL}{\mathcal{L}}
\newcommand{\cX}{\mathcal{X}}

\newcommand{\cC}{\mathcal{C}}
\newcommand{\cK}{\mathcal K}

\newcommand{\cU}{\mathcal U}
\newcommand{\cE}{\mathcal{E}}
\newcommand{\E}{\mathcal{E}}
\newcommand{\cF}{\mathcal{F}}
\newcommand{\pr}{\operatorname{pr}}
\newcommand{\Cone}{\text{\mancone}}
\newcommand{\cV}{\mathcal{V}}
\newcommand{\ZZ}{\mathbb{Z}}
\newcommand{\HH}{\begin{bf}H\end{bf}}
\newcommand{\Gr}{\mathbb{G}\mbox{r}}
\newcommand{\NN}{\begin{bf}N\end{bf}}
\newcommand{\ee}{\begin{bf}e\end{bf}}
\newcommand{\blambda}{{\boldsymbol{\lambda}}}
\newcommand{\balpha}{{\boldsymbol{\alpha}}}

\newcommand{\Chow}{\begin{CJK*}{UTF8}{gbsn}周\end{CJK*}}
\DeclareMathOperator{\chow}{\mbox{\Chow}}

\DeclareMathOperator{\rk}{rk}

\DeclareMathOperator{\sHom}{\mathcal{H}\kern -.5pt\mathit{om}}
\DeclareMathOperator{\sTor}{\mathcal{T}\kern -1.5pt\mathit{or}}

\begin{document}

\date{\today}

\author[A. B. Day]{Andy B. Day}
\address{Department of Mathematics, The Pennsylvania State University, University Park, PA 16802}
\email{andyday@psu.edu}

\author[N. Raha]{Neelarnab Raha}
\address{Department of Mathematics, The Pennsylvania State University, University Park, PA 16802}
\email{neelraha@psu.edu}


\title{Algebraic hyperbolicity of subvarieties of homogeneous varieties}

\begin{abstract}
We study the algebraic hyperbolicity of certain subvarieties of homogeneous varieties, building on the techniques introduced by Coskun-Riedl, Yeong and Mioranci. This generalizes earlier known results for hypersurfaces to higher codimensions. In particular, we observe that if $X=X_1\cap\cdots\cap X_k$ is a very general complete intersection of degree $d_j$ hypersurfaces $X_j$ in $\PP^n$ with $k\leq n-2$, then $X$ is algebraically hyperbolic if $\sum d_j\ge 2n-k$, and $X$ is not algebraically hyperbolic if $\sum d_j\le 2n-k-2$.
\end{abstract}
\maketitle

\setcounter{tocdepth}{1}
\tableofcontents

\section{Introduction}\label{sec:intro}

There are several notions of measuring how far an algebraic variety $X$ is from being isomorphic to projective space. For instance, $X$ is said to be \emph{rational} (resp. \emph{unirational)} if it is birational to some $\PP^n$ (resp., there is a dominant rational map $\PP^n\dashrightarrow X$). Another such measure is \emph{algebraic hyperbolicity}. We say that $X$ is \emph{algebraically hyperbolic} if there is an ample divisor $H$ on $X$ and a positive real number $\epsilon$ (possibly depending on $H$) such that for every integral curve $C$ on $X$, we have the inequality \begin{align}\label{ineq:alg_hyp_definition}
    2g(C)-2\ge\epsilon\deg(HC),
\end{align} where $g(C)$ denotes the geometric genus of $C$. If $X=\PP^n$, then $X$ contains lines and elliptic curves (if $n>1$), which violate the inequality (\ref{ineq:alg_hyp_definition}). Hence, algebraic hyperbolicity is indeed a way of detecting the complexity of $X$.

As the name suggests, algebraic hyperbolicity is the algebraic analog of \emph{Kobayashi} and \emph{Brody} hyperbolicity. Let $X$ be a complex manifold. We say that $X$ is \emph{Kobayashi hyperbolic} if the Kobayashi pseudometric is nondegenerate, and we say that $X$ is \emph{Brody hyperbolic} if there is no non-constant entire map $\CC\to X$. These two notions of hyperbolicity are equivalent if $X$ is compact (see \cite{Brody78}). Demailly introduced the algebraic analog, and proved that for a smooth complex projective variety, Kobayashi hyperbolicity implies algebraic hyperbolicity (see \cite{Demailly97}). The converse is open, and was conjectured by Demailly.

It is often difficult to answer whether a given algebraic variety is algebraically hyperbolic. Nevertheless, there has been significant progress in this area. For example, due to the work of several authors, the following facts are known about degree $d$ hypersurfaces $X$ in $\PP^n$: \begin{enumerate}
    \item (\cite{Xu94,Coskun-Riedl:Quintic}) Suppose $n=3$. If $d\geq5$ and $X$ is very general, then $X$ is algebraically hyperbolic. If $d\leq4$, then $X$ is not algebraically hyperbolic (as it contains rational curves).
    \item (\cite{Clemens86,Clemens86Erratum,Ein:ci01,Voisin96,Voisin96Erratum,Pacienza04,Clemens-Ran04,Clemens-Ran_Erratum,Yeong:HypsrfProdProjSpace}) Suppose $n=4$. If $d\geq7$ and $X$ is very general, then $X$ is algebraically hyperbolic. If $d\leq5$, then $X$ is not algebraically hyperbolic (as it contains lines). Now suppose $n\geq5$. If $d\geq2n-2$ and $X$ is very general, then $X$ is algebraically hyperbolic. If $d\leq2n-3$, then $X$ contains lines and is therefore not algebraically hyperbolic.
\end{enumerate} Hence, the only case that remains open is that of sextic threefolds in $\PP^4$.

There has also been a lot of progress regarding the algebraic hyperbolicity of surfaces in toric threefolds (see \cite{Haase-Ilten21,Robins23,Coskun-Riedl:VGenSurf}), and hypersurfaces in homogeneous varieties (see \cite{Yeong:HypsrfProdProjSpace,Mioranci:HypsrfHomogenVars_v3,Moraga-Yeong24,KwonSeo}).

In \cite{Coskun-Riedl:VGenSurf}, Coskun and Riedl develop a general technique for analyzing the algebraic hyperbolicity of the vanishing locus of a very general global section of a sufficiently ample vector bundle on an algebraic variety that admits a group action. The cases that $\cE$ has either high or low degrees are fairly straightforward. However, the boundary cases are much more difficult, as the anticanonical bundle does not possess strong positivity or negativity properties in these cases. A crucial technique from \cite{Coskun-Riedl:VGenSurf} is to construct certain surface scrolls that help bound the degrees of curves on $X$. This technique has been utilized to study the algebraic hyperbolicity of very general hypersurfaces in products of projective spaces (\cite{Yeong:HypsrfProdProjSpace}) and, more generally, in homogeneous varieties (\cite{Mioranci:HypsrfHomogenVars_v3}). Similar methods have also been used in \cite{Seo:Surfaces_in_Fano_3folds} to study surfaces in Fano threefolds of Picard number $1$.

\subsection*{What we do in this paper} We study the algebraic hyperbolicity of subvarieties $X$ of homogeneous varieties $A\subset\PP^{N_1}\times\cdots\times\PP^{N_m}$. We adapt the scroll construction method referred to above to this situation by developing a technique to suitably increase the dimension of the scrolls and thus bound the degrees of curves on $X$. Our results generalize earlier known results for hypersurfaces in $A$ to higher codimensions. Before stating our main results, we need to introduce some notation.

We let $A$ be a $D$-dimensional homogeneous subvariety of $\PP^{N_1}\times\cdots\times\PP^{N_m}$, and let $\cE$ be a globally generated vector bundle of rank $k\leq D-2$ on $A$. Let $L_1,\ldots,L_m$ be the pullbacks of $\OO_{\PP^{N_1}}(1),\ldots,\OO_{\PP^{N_m}}(1)$ via the natural projections $\PP^{N_1}\times\cdots\times\PP^{N_m}\twoheadrightarrow\PP^{N_i}$ restricted to $A$, and let $H_i$ be the divisor of $L_i$. Suppose that $L_1,\ldots,L_m$ form a \emph{section-dominating collection} of line bundles (see Definition \ref{def:sect_dom}) for $\cE$. We also assume that the canonical divisor class of $A$ can be written as $K_A=a_1H_1+\cdots+a_mH_m$, and that the first and the top Chern classes of $\E$ can be written as $c_1(\cE)=d_1H_1+\cdots+d_mH_m$ and $c_k(\cE)=\sum_{\balpha}d_{\balpha}H_1^{\alpha_1}\cdots H_m^{\alpha_m}$ respectively, where $\balpha=(\alpha_1,\ldots,\alpha_m)$ runs over appropriate multi-indices in $\ZZ_{\geq0}^m$. Finally, for each $1\leq i\leq m$, we let $\ee_i$ be the $i^{\mbox{\tiny th}}$ standard basis vector in $\ZZ^m$.

We are now ready to state our results. Our main result provides bounds on the coefficients of $c_1(\E)$ and $c_k(\E)$ for the vanishing locus $X$ of a very general global section of $\cE$ to be algebraically hyperbolic. Henceforth, we assume that $X$ is smooth and irreducible of codimension $k$ in $A$.

\begin{theorem}\label{thm:Intro_Main_Thm}
    (Theorem \ref{thm:alg_hyp_initial_bound_general} and Theorem \ref{thm:alg_hyp_scroll_bound_general}) Suppose that \begin{align*}
        d_i\geq D-k-a_i-1
    \end{align*} for all $i=1,\ldots,m$, and that at least one of the following is true: \begin{enumerate}
        \item The above inequality is strict for all $i$.
        \item For all $r=1,\ldots,m$, there exists an $m$-tuple $\blambda=(\lambda_1,\ldots,\lambda_m)$ of nonnegative integers with $\lambda_1+\cdots+\lambda_m=k-1$ such that $d_{\blambda+\ee_r}\neq0$ and \begin{align*}
            a_i+d_i-\frac{d_{\blambda+\ee_i}}{d_{\blambda+\ee_r}}>0
        \end{align*} for all $i\ne r$.
    \end{enumerate} Then $X$ is algebraically hyperbolic.
\end{theorem}

Performing a dimension count, we also show that if the coefficients of $c_1(\E)$ are small, then $X$ is not algebraically hyperbolic. More precisely, we have the following result.

\begin{theorem}
    (Theorem \ref{thm:non_alg_hyp_bound_general}) $X$ is not algebraically hyperbolic if $d_i\le D-k-a_i-3$ for some $1\leq i\leq m$.
\end{theorem}

For the above theorem, we do not need to assume that $L_1,\ldots,L_m$ is a section-dominating collection for $\E$, and also $k$ may be equal to $D-1$.

Applying our above results to the case of vector bundles on $\PP^n$ that split into a direct sum of line bundles, we arrive at the following result, which sharpens Ein's result (see \cite{Ein:ci01,Ein:ci02} and \cite[Remark 3.1]{Yeong:HypsrfProdProjSpace}).

\begin{corollary}\label{IntroCorollary_ci_in_Pn}
    (Corollary \ref{cor:ci_Pn}) If $X=X_1\cap\cdots\cap X_k$ is a very general complete intersection in $\PP^n$ with $k\leq n-2$, where $X_j$ is a very general hypersurface of degree $d_j$ for each $j=1,\ldots,k$, then $X$ is algebraically hyperbolic if $\sum_j d_j\geq2n-k$, and $X$ is not algebraically hyperbolic if $\sum_j d_j\leq2n-k-2$.
\end{corollary}

Moraga and Yeong have recently made the following conjecture.

\begin{conjecture}
    (\cite[Conjecture 1.1]{Moraga-Yeong24}) Let $Y$ be a smooth projective variety of dimension at least $2$, and let $L$ be an ample line bundle on $Y$. Then a very general element of the linear system $|K_Y+(3\dim(Y)+1)c_1(L)|$ is an algebraically hyperbolic variety.
\end{conjecture}

Suppose that $a_i\geq-D-1$ for each $1\leq i\leq m$ (recall that in our setup, $D=\dim(A)$ and $K_A=a_1H_1+\cdots+a_mH_m$). Consider the vector bundle \begin{align*}
    \E':=\bigoplus_{j=1}^{k-1}\bigotimes_{i=1}^mL_i^{\otimes d_{i,j}}
\end{align*} on $A$ for some positive integers $d_{i,j}$ with $k\leq D-2$, and let $Y$ be the zero locus of a very general global section of $\E'$. Assume that $Y$ is of codimension $k-1$ in $A$, so that $Y$ is a complete intersection in $A$. Let $\LL$ be an ample line bundle of the form $\bigotimes_{i=1}^mL_i^{\otimes b_i}$ on $A$ for some positive integers $b_1,\ldots,b_m$. Let $\E:=\E'\oplus\LL$, and let $X$ be the zero locus of a very general global section of $\E$. We assume that $X$ is smooth and irreducible of codimension $k$ in $A$. Then Theorem \ref{thm:Intro_Main_Thm}(1) verifies the above conjecture for $Y$ and $\LL|_Y$, under the assumption that $L_1,\ldots,L_m$ form a section-dominating collection of line bundles for $\E$.

In particular, if $Y$ is a complete intersection of codimension $k-1$ in $A=\PP^n$ with $k\le n-2$, then by \cite[Theorem A]{Badescu} (also see \cite{SGA2,Hartshorne:AmpleSubvars}) any ample line bundle on $Y$ is of the form $\cO_Y(d)$ for some positive integer $d$. Thus, the above conjecture is true for $Y$ due to Ein's result.

\subsection*{Organization of the paper} In \S\ref{sec:prelim}, we introduce some notation and our setup, and recall some basic definitions and facts about \emph{Lazarsfeld-Mukai bundles} and curves on subvarieties of homogeneous varieties. In \S\ref{sec:gen_subvar}, we prove our bounds regarding the algebraic hyperbolicity of subvarieties of homogeneous varieties. In \S\ref{sec:comp_int}, we apply those results to the case of complete intersection subvarieties.

\subsection*{Acknowledgments} We would like to thank Izzet Coskun, Jack Huizenga, John Lesieutre, Eric Riedl, Sharon Robins, Ashutosh Roy Choudhury, Julie Tzu-Yueh Wang, Wern Yeong, and several graduate students at Penn State (Andr\'{e}s Gomez-Colunga, Satwata Hans, Eugene Henninger-Voss, and Pisya Vikash) for stimulating and helpful discussions.

\section{Preliminaries}\label{sec:prelim}

We shall work over $\CC$ throughout this paper. In this section, we describe our general setup and collect some basic facts and definitions. We begin by fixing some notation.

\subsection{Notational conventions}\label{subsec:notation}

\begin{itemize}
    \item We shall denote the Chow ring of a variety $Z$ as $\chow(Z)$, since $\chow$ is the Chinese character for Chow. If no confusion is likely to arise, we shall abuse notation and write the class $[Z']$ of a subvariety $Z'\subset Z$ in $\chow(Z)$ as $Z'$ itself.
    \item Given positive integers $N_1,\ldots,N_m$, the $i^{\mbox{\tiny th}}$ projection $\PP^{N_1}\times\cdots\times\PP^{N_m}\twoheadrightarrow\PP^{N_i}$ will be denoted by $\pr_i$. The pullback of the hyperplane class on $\PP^{N_i}$ via $\pr_i$ will be written as $H_i$, and the line bundle $\OO(H_i)$ will be denoted by $L_i$. For any subvariety $Z$ of $\PP^{N_1}\times\cdots\times\PP^{N_m}$, we shall abuse notation and write (the class of) $H_i|_Z$ as $H_i$ itself, and the line bundle $L_i\otimes\OO_Z$ as $L_i$ itself. Furthermore, we denote the ample class $H_1+\cdots+H_m$ on $Z$ by $H$.
    \item Given any $m$-tuple $\blambda=(\lambda_1,\ldots,\lambda_m)$ of nonnegative integers, we shall write \begin{align*}
        |\blambda|:=\lambda_1+\cdots+\lambda_m,
    \end{align*} and $\HH^{\blambda}$ will denote the class $$H_1^{\lambda_1}\cdots H_m^{\lambda_m}$$ in $\chow(Z)$, where $Z$ is any subvariety of $\PP^{N_1}\times\cdots\times\PP^{N_m}$. The $m$-tuple $(N_1,\ldots,N_m)$ will be denoted by $\NN$. As usual, $\ee_i$ will denote the $i^{\mbox{\tiny th}}$ standard basis vector in $\ZZ^m$, for $i=1,\ldots,m$.
\end{itemize}

\subsection{Section-dominating line bundles}

We now describe the notion of \emph{section-dominating line bundles}, as in \cite{Coskun-Riedl:VGenSurf}.

\begin{definition}\label{def:sect_dom}
    (\cite[Definition 2.3]{Coskun-Riedl:VGenSurf}) Given a vector bundle $\cF$ on a smooth projective variety $Z$, we say that a collection $\LL_1,\ldots,\LL_n$ of non-trivial globally generated line bundles on $Z$ is a \emph{section-dominating collection of line bundles for }$\cF$ if each $\LL_i^{\vee}\otimes\cF$ is globally generated and the natural map \begin{align*}
    \bigoplus_{i=1}^n\left(H^0(\cI_p\otimes\LL_i)\otimes H^0(\LL_i^{\vee}\otimes\cF)\right)\to H^0(\cI_p\otimes\cF)
\end{align*} is surjective for each $p\in Z$.
\end{definition}

\begin{example}\label{eg:sect_dom_basic_example}(\cite[Examples 2.4 and 2.6]{Mioranci:HypsrfHomogenVars_v3}).
     The line bundles $L_1=\OO(H_1),\ldots,L_m=\OO(H_m)$ on $Z:=\PP^{N_1}\times\cdots\times\PP^{N_m}$ form a section-dominating collection for $\cF:=d_1L_1\otimes\cdots\otimes d_mL_m$ for any positive integers $d_1,\ldots,d_m$.
\end{example}

\subsection{Lazarsfeld-Mukai bundles}

\begin{definition}
    Given any vector bundle $\cF$ on a smooth projective variety $Z$ with $H^0(\cF)\neq0$, the \emph{Lazarsfeld-Mukai bundle} $M_{\cF}$ associated to $\cF$ is defined to be the kernel of the canonical evaluation map $H^0(\cF)\otimes\OO\to\cF$.
\end{definition}

In particular, if $\cF$ is globally generated, then we have an exact sequence \begin{align*}
    0\to M_{\cF}\to H^0(\cF)\otimes\OO\to\cF\to0
\end{align*} of vector bundles on $Z$.

Note that if $\cK$ is any subbundle of a direct sum of Lazarsfeld-Mukai bundles, then $\deg(\cK)\leq0$ because $\cK$ is a subbundle of a trivial bundle.

Lazarsfeld-Mukai bundles associated to a section-dominating collection of line bundles form a somewhat simpler collection of vector bundles that can be used to bound the degrees of curves on subvarieties. A crucial result in this direction is the following.

\begin{proposition}\label{prop:sect_dom_surjection_general}(\cite[Proposition 2.7]{Coskun-Riedl:VGenSurf}).
    Let $\cF$ be a globally generated vector bundle on a smooth projective variety $Z$, and let $\LL_1,\ldots,\LL_n$ be a section-dominating collection of line bundles for $\cF$. Then there is a surjection \begin{align*}
        \bigoplus_{i=1}^n M_{\LL_i}^{\oplus s_i}\twoheadrightarrow M_{\cF}
    \end{align*} for some nonnegative integers $s_1,\ldots,s_n$.
\end{proposition}

\subsection{Our setup}\label{subsec:setup}

Throughout, we work with a $D$-dimensional smooth projective subvariety $A$ of $\PP^{N_1}\times\cdots\times\PP^{N_m}$ for some positive integers $N_1,\ldots,N_m$. We assume that $A$ is homogeneous with a transitive action by an algebraic group $G$, and that the canonical divisor class of $A$ can be written as $$K_A=a_1H_1+\cdots+a_mH_m.$$ 

We let $\cE$ be a globally generated vector bundle of rank $k$ on $A$ with $1\leq k\le D-2$. We shall assume that \begin{align*}
    L_1=\OO(H_1),\ldots,L_m=\OO(H_m)
\end{align*} form a section-dominating collection of line bundles for $\E$. We also assume that the first and the top Chern classes of $\E$ can be expressed as $$c_1(\E)=d_1H_1+\cdots+d_mH_m\mbox{ and }c_k(\E)=\sum_{|\balpha|=k}d_\balpha\HH^\balpha$$ respectively, and that the vanishing locus $X$ of a very general global section of $\cE$ is a smooth, irreducible subvariety of codimension $k$ in $A$. Our goal in this paper is to investigate the algebraic hyperbolicity of $X$. Note that the class of $X$ in $\chow(A)$ is $c_k(\E)$, and by the tangent normal exact sequence
\begin{align*}
    0\to T_X\to T_A|_X\to\E|_X\to0,
\end{align*}
the canonical class of $X$ is $$K_X=K_A\cdot X+c_1(\E|_X)=(a_1+d_1)H_1+\cdots+(a_m+d_m)H_m.$$ Let $\cX_0\to H^0(\cE)$ be the universal vanishing locus. Suppose that a very general fiber $X$ of $\cX_0$ contains a curve $C$ of geometric genus $g$ and degree $e$ with respect to the ample class $H$ on $X$, i.e., $H\cdot C=e$. We have a relative Hilbert scheme $\mcH\to H^0(\cE)$ with universal curve $\cC_0\to\cX_0$, with the general fiber of $\cC_0\to\mcH$ being a curve of geometric genus $g$ and $H$-degree $e$. By a general argument, there is a $G$-invariant subvariety $\cU$ of $\mcH$ such that the map $\cU\to H^0(\cE)$ is \'etale. Restricting $\cC_0$ to $\cU$, taking a resolution of the general fiber and possibly further shrinking $\cU$ to $\cV$, we get a smooth family $\cC\to\cV$ whose general fiber is a smooth curve of genus $g$. Pulling back $\cX_0$, we get a family $\cX\to\cV$ with a natural, generically injective map $\cC\to\cX$. Let $\psi_1:\cX\to\cV$ and $\psi_2:\cX\to A$ be the natural projections. We refer the reader to \cite{Coskun-Riedl:VGenSurf,Pacienza03,Moraga-Yeong24}. For very general $t\in\cV$, we let $X=X_t$ and $C=C_t$ be the fibers of $\cX$ and $\cC$ respectively over $t$.

By the proof of \cite[Proposition 2.1]{Coskun-Riedl:VGenSurf}, it follows that there is a surjection $M_{\cE}|_C\twoheadrightarrow N_{C/X}$. Combining this with Proposition \ref{prop:sect_dom_surjection_general}, we get a surjection \begin{align}\label{eq:sect_dom_surjection_normal_bundle}
        \bigoplus_{i=1}^m M_{L_i}|_C^{\oplus s_i}\twoheadrightarrow N_{C/X}
    \end{align} on $C$ for some nonnegative integers $s_1,\ldots,s_n$. In fact, one might be able to choose smaller values for the $s_i$'s than those given by Proposition \ref{prop:sect_dom_surjection_general}. This leads us to the following definition.
    
\begin{definition}
    (\cite[Definition 3.2]{Yeong:HypsrfProdProjSpace}) We say that $C\subset X$ is a curve \emph{of type} $(s_1,\ldots,s_m)\in\ZZ_{\geq0}^m$ if there is a surjective map \begin{align*}
        \bigoplus_{i=1}^m M_{L_i}|_C^{\oplus s_i}\twoheadrightarrow N_{C/X},
    \end{align*} with no summand having torsion image.
\end{definition}

Note that if $C\subset X$ is a curve of type $(s_1,\ldots,s_m)$, then \begin{align}\label{ineq:s_up_bound}
    \sum_{i=1}^m s_i\leq\rk(N_{C/X})=D-k-1.
\end{align} We should also point out that the type of $C$ is not necessarily unique.

We now recall a preliminary lemma bounding the genus of the curve $C$ in $X$ which will be used later in this paper. An intermediate step in the proof of this lemma is to bound the degree of the normal bundle of $C$ in $X$.

\begin{lemma}
    (\cite[Lemma 2.8]{Mioranci:HypsrfHomogenVars_v3}) Let $C\subset X$ be a curve of geometric genus $g$ and type $(s_1,\ldots,s_m)$. Then \begin{align}\label{ineq:2g-2>=degN}
        2g-2\ge\sum_{i=1}^m(a_i+d_i-s_i)\deg(c_1(L_i|_C)).
    \end{align}
\end{lemma}

\begin{proof}
    Consider the tangent normal sequence \begin{align*}
    0\to T_C\to T_X|_C\to N_{C/X}\to0
\end{align*} of $C$ in $X$. It shows that \begin{align*}
    2g-2=&\;-\deg T_C\\
    =&\;-\deg(T_X|_C)+\deg N_{C/X}\\
    =&\;\deg(K_X|C)+\deg N_{C/X}\\
    =&\;\sum_{i=1}^m(a_i+d_i)\deg(c_1(L_i|_C))+\deg N_{C/X}\tag{5}\label{ineq:2g-2>=degN_detailed}\\
    \ge&\sum_{i=1}^m(a_i+d_i-s_i)\deg(c_1(L_i|_C)),
\end{align*} where the last inequality can be justified as follows.

Since $C$ is of type $(s_1,\ldots,s_m)$, there is a surjection \begin{align*}
        \gamma:\bigoplus_{i=1}^m M_{L_i}|_C^{\oplus s_i}\twoheadrightarrow N_{C/X}.
    \end{align*} Thus, \begin{align*}
        \deg N_{C/X}=&\;\deg\left(\bigoplus_{i=1}^m M_{L_i}|_C^{\oplus s_i}\right)-\deg(\ker\gamma))\\
        \ge&\;\deg\left(\bigoplus_{i=1}^m M_{L_i}|_C^{\oplus s_i}\right)\\
        =&\;\sum_{i=1}^m s_i\deg(M_{L_i}|_C)\\
        =&\;-\sum_{i=1}^m s_i\deg(c_1(L_i|_C)).
    \end{align*}
\end{proof}

\section{General subvarieties}\label{sec:gen_subvar}

Let $X$ and $C$ be as in \S\ref{subsec:setup}. Suppose that $C$ is of type $(s_1,\ldots,s_m)$. By inequality (\ref{ineq:s_up_bound}), each $s_i\le D-k-1$. So the inequality (\ref{ineq:2g-2>=degN}) becomes
\begin{align*}
    2g-2\ge\sum_{i=1}^m(a_i+d_i-D+k+1)\deg(c_1(L_i|_C)).
\end{align*}
This leads to the following result.
\begin{theorem}\label{thm:alg_hyp_initial_bound_general}
    $X$ is algebraic hyperbolic if $d_i>D-k-a_i-1$ for all $i=1,\ldots,m$.
\end{theorem}

Now, we consider the situation where we allow equalities in the bound, i.e., $d_i\ge D-k-a_i-1$ for all $i$. Fix a smooth curve $C$ of genus $g$ and type $(s_1,\ldots,s_m)$. Suppose $s_i\le D-k-2$ for all $i$. Then from the argument above we have
\begin{align*}
    2g-2\ge\sum_{i=1}^m(a_i+d_i-s_i)\deg(c_1(L_i|_C))\ge\sum_{i=1}^m\deg(c_1(L_i|_C)).
\end{align*}

Since $s_1+\cdots+s_m\le D-k-1$, the remaining case is $s_r=D-k-1$ for some $r$, and $s_i=0$ for all $i\ne r$. For a curve of this type, we adapt the scroll construction from \cite{Yeong:HypsrfProdProjSpace} as follows.

\subsection{Scroll construction}

\setcounter{equation}{5}

Let $X$ and $C$ be as above. By definition of $(s_1,\ldots,s_m)$ we have the surjection
\begin{align*}
    \gamma:M_{L_r}|_C^{\oplus D-k-1}\twoheadrightarrow N_{C/X}
\end{align*}
with torsion free image. Due to \cite{Clemens03,Clemens-Ran04,Clemens-Ran_Erratum}, we know that $\gamma$ induces a map $$M_{L_r}|_C\twoheadrightarrow N_{C/X}/(\gamma(M_{L_r}^{D-k-2}))$$ with rank one image $Q$. The exact sequence
\begin{align*}
    0\to M_{L_r}|_C^{\oplus D-k-2}\xrightarrow{\gamma}N_{C/X}\to Q\to0
\end{align*}
gives us
\begin{align}\label{N=Q}
    c_1(N_{C/X})=-(D-k-2)c_1(L_r|_C)+c_1(Q).
\end{align}

Following the scroll construction from \cite{Yeong:HypsrfProdProjSpace}, we use the diagram
    \begin{align*}\label{}
    \xymatrixcolsep{7pc}
    \xymatrix@R3pc@!0{
            0\ar[rd]&0\ar[d]&&&\\
            &K\ar[d]\ar[rd]&&0&\\
            0\ar[r]&M_{L_r}|_C\ar[r]\ar[d]&H^0(L_r)\otimes\OO_C\ar[r]\ar[rd]&L_r|_C\ar[r]\ar[u]&0\\
            &Q\ar[d]&&Q'\ar[rd]\ar[u]&\\
        &0&&&0 
    }
    \end{align*}
to construct an irreducible two dimensional $\PP^{N_r}$-scroll $S$ over $C$ induced by the rank $2$ bundle $Q'$. That is, $S$ is a one parameter union of lines of class $$\HH^{\NN-\ee_r}=H_1^{N_1}\cdots H_{r-1}^{N_{r-1}}H_r^{{N_r-1}}H_{r+1}^{N_{r+1}}\cdots H_m^{N_m}$$ in $\chow(\PP^{N_1}\times\cdots\times\PP^{N_r})$. Furthermore, the above diagram yields
\begin{align*}
    [S] = (\deg(c_1(L_r|_C)+c_1(Q)))\HH^{\NN-2\ee_r}+\sum_{i\ne r}(\deg(c_1(L_i|_C)))\HH^{\NN-\ee_r-\ee_i}
\end{align*} in the same Chow ring.

To bound $\deg c_1(Q)$ by intersecting $S$ and $X$, we introduce the join construction to increase the dimension of $S$ as follows.

\begin{definition}
    For any subvariety $Z\subset\PP^{N_1}\times\cdots\times\PP^{N_m}$, the $\PP^{N_i}$-cone of $Z$, denoted by $\Cone_i(Z)$, is defined to be the top dimension part of the closure of
    $$\pi_i^{-1}\pi_iZ,$$
    where $\pi_i:\PP^{N_1}\times\cdots\times\PP^{N_m}\to\PP^{N_1}\times\PP^{N_{i-1}}\times\PP^{N_i-1}\times\PP^{N_{i+1}}\times\cdots\times\PP^{N_m}$ is the projection from $\PP^{N_1}\times\PP^{N_{i-1}}\times\{p\}\times\PP^{N_{i+1}}\times\cdots\times\PP^{N_m}$ for some generic $p\in\PP^{N_i}$.
\end{definition}

\begin{definition}
    Given an $m$-tuple $ \blambda=(\lambda_1,\ldots,\lambda_m)$ of nonnegative integers and a subvariety $ Z\subset\PP^{N_1}\times\cdots\times\PP^{N_m}$, we define the $ \blambda$-join of $Z$ to be
    $$\Join_ \blambda Z:=\Cone_1^{\lambda_1}\cdots\Cone_m^{\lambda_m}(Z),$$
    where $\Cone_i^{\lambda_i}$ means repeating the cone operator $\Cone_i$ $\lambda_i$ times.
\end{definition}

While the cone $\Cone_i(Z)$ and $ \blambda$-join $\Join_ \blambda Z$ of a variety $Z$ are not unique, their classes in the Chow ring $\chow(\PP^{N_1}\times\cdots\times\PP^{N_m})$ are uniquely given by the following formula.

\begin{proposition}
    Let $Z\subset \PP^{N_1}\times\cdots\times \PP^{N_m}$ be a subvariety of class $\sum_{\balpha}b_\balpha\HH^\balpha$, and let $\blambda$ be an $m$-tuple of nonnegative integers. Then
    $$[\Join_\blambda Z]=\sum_\balpha b_\balpha\HH^{\balpha-\blambda},$$
    where $\HH^{(\beta_1,\ldots,\beta_m)}:=0$ if $\beta_i<0$ for any $i$.
\end{proposition}

Now, let $S$ be the scroll constructed above, then the above proposition gives us
\begin{align*}
    \Join_\blambda S=(\deg(c_1(L_r|_C)+c_1(Q)))\HH^{\NN-2\ee_r-\blambda}+\sum_{i\ne r}(\deg(c_1(L_i|_C)))\HH^{\NN-\ee_r-\ee_i-\blambda}.
\end{align*}

Suppose $|\blambda|=k-1$. Notice that since $S$ is irreducible, $C\subset(\Join_\blambda S)\cap X\subset A$. So by taking intersection product in $\chow(A)$, we have
\begin{align*}
    0\le\deg(((\Join_\blambda S) X-C)H_r)=\deg\!\left(d_{\blambda+\ee_r}(c_1(L_r|_C)+c_1(Q))+\sum_{i\ne r}d_{\blambda+\ee_i}c_1(L_i|_C)-c_1(L_r|_C)\right)\!.
\end{align*}
Assuming $d_{\blambda+\ee_r}\neq0$, we see that \begin{align*}
    \deg c_1(Q)\ge\deg\!\left(\left(\frac{1}{d_{\blambda+\ee_r}}-1\right)c_1(L_r|_C)-\sum_{i\ne r}\frac{d_{\blambda+\ee_i}}{d_{\blambda+\ee_r}}c_1(L_i|_C)\right)\!.
\end{align*}

Together with (\ref{ineq:2g-2>=degN_detailed}) and (\ref{N=Q}), we arrive at the following theorem.

\begin{theorem}\label{thm:alg_hyp_scroll_bound_general}
    Let $A$, $\E$, $X$, $d_\balpha$'s and $d_i$'s be as above, and $d_i\ge D-a_i-k-1$ for all $i$. Also suppose that for all $r=1,\ldots,m$, there exists an $m$-tuple $\blambda$ of nonnegative integers with $|\blambda|=k-1$ such that $d_{\blambda+\ee_r}\neq0$ and \begin{align*}
        a_i+d_i-\frac{d_{\blambda+\ee_i}}{d_{\blambda+\ee_r}}>0
    \end{align*} for all $i\ne r$. Then $X$ is algebraically hyperbolic.
\end{theorem}

\subsection{Non-hyperbolicity in low degrees}
On the other hand, we have the following upper bound for non-hyperbolicity. We do not assume that $\E$ is section-dominated by $L_1,\ldots,L_m$ for this result. Also, $k$ may be equal to $D-1$.

\begin{theorem}\label{thm:non_alg_hyp_bound_general}
	Let $A$, $\E$, $X$ and $d_i$'s be as above, with the exception that $\E$ need not be section-dominated by $L_1,\ldots,L_m$. Suppose that \begin{align*}
		d_i\le D-k-a_i-3
	\end{align*} for some $1\leq i\leq m$. Then $X$ is not algebraically hyperbolic.
\end{theorem}

\begin{proof}
	Fix $1\leq i\leq m$, we show that the inequality in the theorem implies that $X$ contains lines. Let \begin{align*}
		F_i(A)\subset\PP^{N_1}\times\cdots\times\PP^{N_{i-1}}\times\Gr(1,N_i)\times\PP^{N_{i+1}}\times\cdots\times\PP^{N_m}
	\end{align*} be the Fano scheme of all $\PP^{N_i}$-lines in $A$. Consider the incidence variety
	\begin{align*}
		\Theta:=\{(\ell,t):\ell\subset X_t\}\subset F_i(A)\times\cV,
	\end{align*}
	of $\PP^{N_i}$-lines contained in the zero locus $X_t$. 
	Let $\pi_{F_i}$ and $\pi_\cV$ be the projection from $\Theta$ to $F_i(A)$ and $\cV$ respectively. Our goal is to determine whether $\pi_\cV$ is surjective.
	
	Since $\E$ is globally generated, $h^1(\E|_\ell)=0$ for any $\PP^{N_i}$-line $\ell\subset A$ as $\cE|_{\ell}$ splits into a direct sum of line bundles of the form $\OO_{\ell}(n)$ with $n\ge0$. Therefore, from the exact sequence
	\begin{align*}
		0\to\E\otimes\cI_\ell\to\E\to\E_\ell\to0
	\end{align*}
	on $A$, we have
	\begin{align*}
		\dim\pi_{F_i}^{-1}(\ell)=&\;h^0(\E\otimes\cI_\ell)\\
		=&\;h^0(\E)-h^0(\E_\ell)+h^1(\E\otimes\cI_\ell)-h^1(\E)\\
		=&\;h^0(\E)-d_i-k+h^1(\E\otimes\cI_\ell)-h^1(\E)
	\end{align*}
	for any $\PP^{N_i}$-line $\ell$ in $A$. In particular, since the long exact sequence ends at $H^1(\E\otimes\cI_\ell)\to H^1(\E)\to0$, we have $h^1(\E\otimes\cI_\ell)-h^1(\E)\ge0$. From \cite[Lemma 3.3]{Mioranci:HypsrfHomogenVars_v3}, we have $\dim F_i(A)=D-a_i-3$. Thus,
	\begin{align*}
		\dim\Theta=\;&\dim F_i(A)+\dim\pi_{F_i}^{-1}(\ell)\;\\
		=\;&D-a_i-3+h^0(\E)-d_i-k+h^1(\E\otimes\cI_\ell)-h^1(\E)\\
		\ge\;&D-a_i-3+h^0(\E)-d_i-k.
	\end{align*}
	Therefore, if \begin{align*}
		D-a_i-3-d_i-k\ge0,
	\end{align*} then \begin{align*}
		\dim\Theta-\dim\cV\ge\dim\Theta-h^0(\E)\ge D-a_i-3-d_i-k\ge0,
	\end{align*} whence $\pi_\cV$ is surjective.
\end{proof}

\section{Complete intersections}\label{sec:comp_int}

In this section, we apply our results to the case where $\E$ splits as a direct sum of line bundles, i.e., 
\begin{align*}
    \E:=\bigoplus_{j=1}^k\bigotimes_{i=1}^mL_i^{\otimes d_{i,j}}
\end{align*} for some positive integers $d_{i,j}$. In this case, $X$ is a complete intersection subvariety of $A$. Note that if $A=\PP^{N_1}\times\cdots\times\PP^{N_m}$, then it is clear from Example \ref{eg:sect_dom_basic_example} that $\E$ is section-dominated by $L_1,\ldots,L_m$.

We recall our notation \begin{align*}
    d_1H_1+\cdots+d_mH_m=c_1(\E)=\left(\sum_{j=1}^k d_{1,j}\right)\!H_1+\cdots+\left(\sum_{j=1}^k d_{m,j}\right)\!H_m
\end{align*} and \begin{align*}
    \sum_{|\balpha|=k}d_\balpha\HH^\balpha=c_k(\E)=[X]=\prod_{j=1}^k\left(\sum_{i=1}^m d_{i,j}H_i\right).
\end{align*}

By Theorems \ref{thm:alg_hyp_initial_bound_general} and \ref{thm:alg_hyp_scroll_bound_general}, we get the following result.

\begin{theorem}\label{thm:alg_hyp_complete}
    Let notation be as just above. Suppose that \begin{align*}
        \sum_{j=1}^k d_{i,j}\ge D-k-a_i-1
    \end{align*} for all $i=1,\ldots,m$, and that at least one of the following is true: \begin{enumerate}
        \item The above inequality is strict for all $i$.
        \item\label{thm:alg_hyp_scroll_complete} For all $r=1,\ldots,m$, there exists an $m$-tuple $\blambda$ of nonnegative integers with $|\blambda|=k-1$ such that \begin{align*}
            a_i+\sum_{j=1}^k d_{i,j}-\frac{d_{\blambda+\ee_i}}{d_{\blambda+\ee_r}}>0
        \end{align*} for all $i\ne r$.
    \end{enumerate} Then $X$ is algebraically hyperbolic.
\end{theorem}

In particular, by choosing $\blambda=(k-1)\ee_r$, we obtain the following special case for complete intersections with uniform degrees. We suppose that for each $j$, $d_{i,j}$ is independent of $i$. Then $d:=d_1=\cdots=d_m$, and $c_1(\E)=dH_1+\cdots+dH_m$.

\begin{corollary}
    Let $d$ be as above. Suppose that $d\ge D-k-a_i-1$ for all $i$, and that $(D-1)/2>k$. Then $X$ is algebraically hyperbolic.
\end{corollary}

On the other hand, due to Theorem \ref{thm:non_alg_hyp_bound_general}, we have the following bound for $X$ to fail to be algebraically hyperbolic.

\begin{theorem}\label{thm:non_alg_hyp_ci}
    Let notation be as above. Suppose that \begin{align*}
        \sum_{j=1}^k d_{i,j}\le D-k-a_i-3
    \end{align*} for some $1\leq i\leq m$. Then $X$ is not algebraically hyperbolic.
\end{theorem}

We would like to remind the reader that for the above theorem, we need not assume that $L_1,\ldots,L_m$ form a section-dominating collection for $\E$, and also that $k$ may be equal to $D-1$.

Applying Theorems \ref{thm:alg_hyp_complete} and \ref{thm:non_alg_hyp_ci}, we obtain the following corollary for $A=\PP^n$.

\begin{corollary}\label{cor:ci_Pn}
    If $X=X_1\cap\cdots\cap X_k$ is a very general complete intersection in $\PP^n$ for some positive integers $n$ and $k$ with $k\leq n-2$, where $X_j$ is a very general hypersurface of degree $d_j$ for each $j=1,\ldots,k$, then $X$ is algebraically hyperbolic if $\sum_j d_j\geq2n-k$, and $X$ is not algebraically hyperbolic if $\sum_j d_j\leq2n-k-2$.
\end{corollary}

\bibliography{mybibOUT}{}
\bibliographystyle{alphaurl}

\end{document}